\documentclass[11pt]{amsart}
\usepackage{amsfonts,amssymb,amscd,amsmath,enumerate,verbatim,calc}
\usepackage{amsmath}
\usepackage{amssymb}
\usepackage{amscd}
\def\NZQ{\mathbb}               % the font for N,Z,Q,R,C
\def\NN{{\NZQ N}}
\def\QQ{{\NZQ Q}}
\def\ZZ{{\NZQ Z}}
\def\RR{{\NZQ R}}
\def\CC{{\NZQ C}}

\def\PP{{\NZQ P}}

\newtheorem{Theorem}{Theorem}[section]
\newtheorem{Lemma}[Theorem]{Lemma}
\newtheorem{Corollary}[Theorem]{Corollary}
\newtheorem{Proposition}[Theorem]{Proposition}
\newtheorem{Remark}[Theorem]{Remark}

\newtheorem{Definition}[Theorem]{Definition}

\let\epsilon\varepsilon
\let\phi=\varphi
\let\kappa=\varkappa

\begin{document}

\title{Asymptotic Regularity of  Powers of Ideals of  Points in a Weighted Projective Plane}

\author{Steven Dale Cutkosky  and Kazuhiko Kurano}
\thanks{2000 Mathematics Subject Classifications: 13A99, 14Q10}
\thanks{The first author was partially supported by NSF}
\thanks{The second author was supported by KAKENHI (21540050)}

\address{Steven Dale Cutkosky, Department of Mathematics,
University of Missouri, Columbia, MO 65211, USA}
\email{cutkoskys@missouri.edu}

\address{Kazuhiko Kurano,  Department of Mathematics,
School of Science and Technology, Meiji University,
Higashimata 1-1-1, Tama-ku, Kawasaki-shi 214-8571, Japan}
\email{kurano@math.meiji.ac.jp}

\maketitle

\begin{center}
{\it
Dedicated to the memory of Professor Masayoshi Nagata}
\end{center}

\begin{abstract}
In this paper we study the asymptotic behavior of the regularity of  symbolic powers of ideals 
of  points in a weighted projective plane.
By a result of Cutkosky, Ein and Lazarsfeld~\cite{CEL},
regularity of such powers behaves asymptotically like a linear function.
We study the difference between regularity of such powers and this linear function.
Under some conditions, we prove that this difference is bounded, or eventually periodic.

As a corollary we show that, if there exists a negative curve, 
 then the regularity of  symbolic powers of a monomial space curve is eventually 
a periodic linear function. 
We  give a criterion for the validity of Nagata's conjecture in terms of the lack of existence of 
negative curves.
\end{abstract}

\section{Introduction}

Suppose that $H$ is an ample $\QQ$-Cartier divisor on a normal projective variety $V$, and 
$\mathcal I$ is an ideal sheaf on $V$. Let $\nu:W\rightarrow V$ be the blow up
of $\mathcal I$. Let $E$ be the effective Cartier divisor on $W$ defined by
${\mathcal O}_W(-E)={\mathcal I}{\mathcal O}_W$. The $s$-invariant, $s_{\mathcal O_V(H)}({\mathcal I})$, is defined by
$$
s_{\mathcal O_V(H)}({\mathcal I})={\rm inf}\{s\in\RR\mid \nu^*(sH)-E\mbox{ is an ample $\RR$-divisor on $W$}\}.
$$
The reciprical, $\frac{1}{s_H({\mathcal I})}$, is the Seshadri constant of ${\mathcal I}$.

Examples in \cite{C1} and \cite{CEL} show that 
$s_H({\mathcal I})$ can be irrational, even when $\mathcal O_V(H)\cong {\mathcal O}_{\PP}(1)$
on ordinary projective space.

Suppose that $K$ is a field, and $a_0,\ldots,a_{\overline n}$ are  positive  integers. Let $S=K[x_0,\ldots,x_{\overline n}]$ be
a polynomial ring, graded by the  weighting $\text{wt}(x_i)=a_i$ for $0\le i\le \overline n$. Let $\mathfrak m$ be the graded maximal ideal of $S$. Let
$\PP=\PP(a_0,a_1,\ldots,a_{\overline n})={\rm proj}(S)$ be the associated weighted projective space. $\PP$ is a normal projective variety. 
$\PP$ is  isomorphic to a weighted projective space in which ${\rm gcd}(a_0,a_1,\ldots,a_{i-1},a_{i+1}\ldots,a_{\overline n})=1$ for $0\le i\le \overline n$ \cite{De}, \cite{Do}.

We will  suppose through most of this paper that $a_0,\ldots,a_{\overline n}\in\ZZ_+$ 
satisfy the condition that ${\rm gcd}(a_0,a_1,\ldots,a_{i-1},a_{i+1}\ldots,a_{\overline n})=1$ for $0\le i\le \overline n$.
With this assumption on the $a_i$, there exists a Weil divisor $H$ on $\PP$ such that $\mathcal O_{\PP}(r)\cong \mathcal O_{\PP}(rH)$ is a divisorial sheaf of $\mathcal O_{\PP}$ modules (reflexive of rank 1) for all $r\in\ZZ$ \cite{Mo}.

Let $M$ be a finitely generated, graded $S$-module. The local cohomology modules $H_{\mathfrak m}^i(M)$ are naturally graded.  The regularity
of $M$ is defined (\cite{EG}) by
$$
{\rm reg}(M)={\rm max}\{i+j\mid H_{\mathfrak m}^i(M)_j\ne 0\}.
$$

Suppose that $I\subset S$ is a homogeneous ideal. Let $I^{{\rm sat}}$ be the saturation of $I$ with respect to the graded maximal ideal of  $I$.
Let $\mathcal I$ be the sheaf associated to $I$ on $\PP$.
Let $X=X(I)={\rm proj}(\bigoplus_{m\ge 0}{\mathcal I}^m)$ be the blow up of $\mathcal I$, with natural
projection $f:X\rightarrow \PP$.

In Section \ref{sec2}, we develop the basic properties of regularity on weighted projective space, and show that the theory of asymptotic  regularity on
ordinary projective space extends naturally to weighted projective space. For instance, the statement on ordinary projective space, proven in Theorem 1.1 of \cite{CHT}, or in \cite{Ko},
extends to show that ${\rm reg}(I^m)$ is a linear function for $m\gg 0$. We also establish  the following basic result, which generalizes the statement for ample line bundles proven in Theorem B of \cite{CEL} to the $\QQ$-Cartier Weil divisor $\mathcal O_{\PP}(1)$ on weighted projective space (with the $a_i$ pairwise relatively prime).

\begin{Theorem}\label{Theoremlim} We have that
$$
\lim_{m\rightarrow \infty} \frac{{\rm reg}((I^m)^{{\rm sat}})}{m}=s_{\mathcal O_{\PP}(1)}({\mathcal I}).
$$
\end{Theorem}

In general, as commented above, this limit is irrational, so ${\rm reg}((I^m)^{{\rm sat}})$ is in general far from being a linear function. 

We will write  $\lfloor x\rfloor$ for the greatest integer in a real number $x$.
We may define a function $\sigma_I:\NN\rightarrow \ZZ$ by
$$
{\rm reg}((I^m)^{{\rm sat}})=\lfloor m s_{\mathcal O_{\PP}(1)}(\mathcal I)\rfloor +\sigma_I(m).
$$
By Theorem \ref{Theoremlim}, we have that
$$
\lim_{m\rightarrow \infty}\frac{\sigma_I(m)}{m}=0.
$$
An interesting question is to determine when $\sigma_I(m)$ is bounded.
We do not know of an example where $\sigma_I(m)$ is not bounded.

In this paper, we study the case where $I\subset S=K[x,y,z]$ is the ideal of a set of nonsingular points with multiplicity (a ``fat point'') in a weighted two dimensional projective space $\PP=\PP(a,b,c)$, with ${\rm wt}(x)=a$, ${\rm wt}(y)=b$, ${\rm wt}(z)=c$ and
$a,b,c$ pairwise relatively prime. We also assume that $K$ is algebraically closed. Suppose that $P_1,\ldots, P_r$ are distinct nonsingular closed points in 
$\PP(a,b,c)$, and $e_i$ are positive integers.
Let $I_{P_i}\subset S=K[x,y,z]$ be the weighted homogeneous ideal of the point
$P_i$, and let $I=\cap_{i=0}^rI_{P_i}^{e_i}$.
Let ${\mathcal I}=\tilde I$ be the sheafication of $I$ on $\PP$,
and define
$$
s(I)=s_{\mathcal O_{\PP}(1)}({\mathcal I}).
$$
Let $u=\sum_{i=1}^re_i^2$.
 We have that $s(I)\ge \sqrt{abcu}$. If 
$s(I)>\sqrt{abcu}$, then $s(I)$ is a rational number.

Nagata's conjecture states that $s(I)=\sqrt{r}$ if $r\ge 9$, $e_i=1$ for $1\le i\le r$ and $P_1,\ldots, P_r$ are independent generic points in ordinary projective space $\PP^2$. Nagata proved this conjecture in \cite{N1} in the case that $r$ is a perfect square, as a critical ingredient in his counterexample to Hilbert's fourteenth problem.
A proof of Nagata's conjecture in the case of an $r$ which is not a perfect square would give a set of points in $\PP^2$ for which $s(I)$ is not rational. Some recent papers on regularity and $s$-invariants of points in $\PP^2$ are \cite{CTV}, \cite{DG}, \cite{Gi}, \cite{HR} and \cite{Hi}.

Let $I^{(m)}$ be the $m$-th symbolic power of $I$, which is also, in our
situation, the saturation $(I^m)^{{\rm sat}}$ of $I^m$ with respect to the
graded maximal ideal $\mathfrak m$ of $S$.

We prove the following asymptotic statements about regularity in Section \ref{sec3}.

A function $\sigma:\NN\rightarrow \ZZ$ is bounded if there exists $c\in\NN$ such that
$|\sigma(m)|<c$ for all $m\in\NN$. In Theorem \ref{Theorem4}, we prove:
\vskip .1truein
\noindent {\it Let $I$ be the ideal of a set of fat points in a weighted projective plane. Then
$\sigma_I(m)$ is a bounded function.} 
\vskip .1truein

If the graded $K$-algebra $\bigoplus_{m\ge 0}I^{(m)}$ is a finitely generated $K$-algebra, then ${\rm reg}(I^{(m)})$ must be  a quasi polynomial for large $m$. 
A quasi polynomial is a polynomial in $m$ with coefficients which are periodic functions in $m$. In general, $\bigoplus_{m\ge 0}I^{(m)}$ is not a finitely generated $K$-algebra. Some examples where this  algebra is not finitely generated are given by Nagata's Theorem \cite{N1}, showing that it is not finitely generated when $r\ge 9$ is a perfect square and $r$ generic points in $\PP^2$ are blown up. Goto, Nishida and Watanabe \cite{GNW} give examples of monomial primes
$P(a,b,c)$ such that the symbolic algebra is not finitely generated.

A function $\sigma:\NN\rightarrow \ZZ$ is {\it eventually periodic} if $\sigma(m)$ is periodic for
 $m\gg 0$. In Theorem \ref{Theorem5}, we prove:
\vskip .1truein
\noindent {\it  Suppose that $s(I)>\sqrt{abcu}$
and $K$ has characteristic zero or is the algebraic closure of a finite field.
Then the function $\sigma_I(m)$  is
eventually periodic.}
\vskip .1truein

An example, defined over a field $K$ which is of positive characteristic and is transcendental over the prime field, where $s(I)>\sqrt{abcu}$  but $\sigma(m)$ is not eventually periodic, is given in Example 4.4 \cite{CHT}. In this  example, constructed from
17 special points $P_i$ in ordinary projective space $\PP^2$,
$I=I_{P_1}\cap \cdots \cap I_{P_{13}}\cap I_{P_{14}}^2\cap \cdots \cap I_{P_{17}}^2$. We have $s(I)=\frac{29}{5}>\sqrt{abcu}=\sqrt{29}$ in this example.

Let $H$ be a Weil divisor on $\PP$ such that $\mathcal O_{\PP}(1)\cong \mathcal O_{\PP}(H)$ and let $A= f^*(H)$.
An effective divisor $D$ on $X$ such that $(D\cdot D)<0$ will be called a negative curve.
An effective divisor $D$ such that $D\sim aA-mE$ for some positive integers $a$ and $m$ will be called an $E$-uniform curve.

We establish in Corollary \ref{Cor6} that
\vskip .1truein
\noindent {\it Suppose  there exists an $E$-uniform negative curve on $X(I)$,
and $K$ has characteristic zero, or is a finite field. Then $s(I)$ is a rational number and the function $\sigma_I(m)$ of Theorem \ref{Theorem4} is
eventually periodic.}
\vskip .1truein

An important case of this construction is when $i=1$, and
$I=P(a,b,c)$ is the prime ideal of a monomial space curve. The ideal $P(a,b,c)$ is defined to be the kernel of the $K$-algebra homomorphism $K[x,y,z]\rightarrow K[t]$, given by $x\mapsto t^a$, $y\mapsto t^b$, $z\mapsto t^c$.
As a corollary to Theorems \ref{Theorem4} and \ref{Theorem5}, we have the following application, Corollary \ref{Cor7*},  to  monomial space curves.
\vskip .1truein

\noindent  {\it Suppose that $I=P(a,b,c)$ is the prime ideal of a monomial space curve, and   there exists a  negative curve
on $X(I)$.
 Then $s(I)$ is a rational number, and the function $\sigma_I(m)$  is
eventually periodic.} 
\vskip .1truein
 
 We do not know of an example of  a monomial prime  $I=P(a,b,c)$ where there does not exist a negative curve. This interesting problem 
 is discussed in \cite{KM}.
 
In Section \ref{sec5}, we  give a criterion for
the validity of Nagata's conjecture in terms of the lack of existence of uniform negative curves on certain weighted projective planes.

\section{Regularity on Weighted Projective Space}\label{sec2}

In this section we define the regularity of
a finitely generated graded module
over a non-standard graded polynomial ring.

Let $K$ be a field and $B = K[x_1, \ldots, x_s]$ be
a graded polynomial ring with ${\rm wt}(x_1)=d_1$, \ldots,
${\rm wt}(x_s)=d_s$, where
$d_1$, \ldots, $d_s$ are positive integers.
Set ${\mathfrak m} = (x_1, \ldots, x_s)B$.

\begin{Definition}
\begin{rm}
For a finitely generated $B$-module $M \neq 0$,
We define $a_i(M)$, ${\rm reg}(M)$, ${\rm reg}_i(M)$ and
${\rm reg}'(M)$ as follows:
\begin{eqnarray}
a_i(M) & = & 
\left\{
\begin{array}{lll}
{\rm max}\{ j \in \ZZ \mid H^i_{\mathfrak m}(M)_j\ne 0 \}
& & \mbox{if $H^i_{\mathfrak m}(M)\ne 0$} \\
-\infty & & \mbox{otherwise}
\end{array}
\right.
\nonumber
\\
{\rm reg}(M) & = & 
{\rm max}\{ i+j \mid H^i_{\mathfrak m}(M)_j\ne 0 \}
= {\rm max}\{ a_i(M)+i \mid 0 \le i \le \dim M \}
\label{eqLC4}
\\
{\rm reg}_i(M) & = & 
\left\{
\begin{array}{lll}
{\rm max}\{ j \in \ZZ \mid 
{\rm Tor}^B_i(M, B/{\mathfrak m})_j\ne 0 \} - i
& & \mbox{if ${\rm Tor}^B_i(M, B/{\mathfrak m})\ne 0$} \\
-\infty & & \mbox{otherwise}
\end{array}
\right.
\nonumber
\\
{\rm reg}'(M) & = & 
{\rm max}\{ {\rm reg}_i(M) \mid i \ge 0 \}
\nonumber
\end{eqnarray}
\end{rm}
\end{Definition}

It is not difficult to prove the following theorem (cf.\ Theorem~3.5 in \cite{DS}).
We omit a proof.

\begin{Theorem}
With notation as above,
\[
{\rm reg}(M) = {\rm reg}'(M) + s - \sum_{i = 1}^sd_i.
\]
\end{Theorem}

\begin{Remark}
\begin{rm}
Assume $d_1 = \cdots = d_s = 1$.
Let $I$ be a homogeneous ideal of $B$.
Put $\PP = {\rm proj}(B)$.

The regularity of a coherent ${\mathcal O}_{\PP}$ module 
${\mathcal F}$ is 
\begin{equation}\label{eqLC5}
{\rm reg}({\mathcal F})=
{\rm min}\{l\mid H^i(\PP,{\mathcal F}(j-i))=0
\mbox{ for all $j\ge l$ and $i>0$}\}.
\end{equation}
Let ${\mathcal I}$ be the ideal sheaf on $\PP$ associated to $I$.
We have that 
\begin{equation}\label{eqL6}
{\rm reg}((I^m)^{\rm sat})={\rm reg}({\mathcal I}^m)
\end{equation}
 for all $m\ge 0$, as follows from Theorem A4.1 \cite{E}.

For each $i \ge 0$, ${\rm reg}_i(I^m)$ is eventually linear
on $m$ by Theorem~3.1 in \cite{CHT}.

Hence, ${\rm reg}(I^m)$ is eventually linear on $m$ as in 
Theorem~1.1 (ii) in \cite{CHT}.

On the other hand, there exists an example that
$a_i(I^m)$ is not eventually linear on $m$ as follows.
Let $I$ be the ideal in Example~4.4 in \cite{CHT} (which was refered to in the introduction).
Then,
\[
{\rm reg}((I^m)^{\rm sat}) =
{\rm max}\{ a_2((I^m)^{\rm sat}) + 2, a_3((I^m)^{\rm sat}) + 3 \}
= {\rm max}\{ a_2(I^m) + 2, 0 \} =a_2(I^m)+2
\]
since ${\rm reg}((I^m)^{\rm sat})={\rm reg}'((I^m)^{\rm sat})>0$,
and ${\rm reg}((I^m)^{\rm sat})$ is  not eventually linear in $m$.
Therefore, $a_2(I^m)$ is not eventually linear in $m$ in this case.
\end{rm}
\end{Remark}

In Section \ref{sec3}, we will consider the case when $S=K[x,y,z]$, with ${\rm wt}(x)=a$, ${\rm wt}(y)=b$ and ${\rm wt}(z)=c$
for pairwise relatively prime positive integers $a,b,c$, and $I=I_{P_1}^{e_1}\cap \cdots\cap I_{P_r}^{e_r}$ with $P_i$ distinct nonsingular points of $\PP(a,b,c)$.   We will 
then have that 
\[
{\rm reg}((I^m)^{\rm sat}) =
{\rm max}\{ a_2((I^m)^{\rm sat}) + 2, a_3((I^m)^{\rm sat}) + 3 \}
= {\rm max}\{ a_2(I^m) + 2, 3 - a-b-c \}=a_2(I^m) + 2
\]
since ${\rm reg}((I^m)^{\rm sat})={\rm reg}'((I^m)^{\rm sat})+3-a-b-c>3-a-b-c$,
and (with the notation defined in the introduction)
\[
a_2(I^m) = {\rm max}\{ n \in \ZZ \mid
H^1(X,\mathcal{O}_X(nA - mE))\ne 0 \} .
\]

\begin{Remark}
\begin{rm}
Let $B_1 = K[x_1, \ldots, x_s]$ and $B_2 = K[y_1, \ldots, y_s]$
be graded polynomial rings 
with ${\rm wt}(x_i)=d_i$ ($i = 1, \ldots, s$) and
${\rm wt}(y_j)=d'_j$ ($j = 1, \ldots, s$), where the
$d_i$'s and $d'_j$'s are positive integers.

Let $\delta : B_1 \rightarrow B_2$ be a flat $K$-algebra graded
homomorphism.
Assume that $B_2/(x_1, \ldots, x_s)B_2$ is of finite length.

Set ${\mathfrak m}_1 = (x_1, \ldots, x_s)B_1$
and ${\mathfrak m}_2 = (y_1, \ldots, y_s)B_2$.

Let $M$ be a finitely generated graded $B_1$-module.
Then,
\begin{equation}\label{cohomology}
H^i_{{\mathfrak m}_1}(M) \otimes_{B_1}B_2
= H^i_{{\mathfrak m}_1B_2}(M\otimes_{B_1}B_2)
= H^i_{{\mathfrak m}_2}(M\otimes_{B_1}B_2) .
\end{equation}
Here, set
\[
\xi = {\rm max}\{ n \in \ZZ \mid [B_2/{\mathfrak m}_1B_2]_n \ne 0 \} .
\]
By (\ref{cohomology}), we obtain
\[
a_i(M\otimes_{B_1}B_2)
= a_i(M) + \xi
\]
for $i = 0, \ldots, \dim M$.
Therefore,
\begin{equation}\label{basechange}
{\rm reg}(M\otimes_{B_1}B_2)
= {\rm reg}(M) + \xi .
\end{equation}
\end{rm}
\end{Remark}

From the above Remark, we see that the statement on ordinary projective space, proven in Theorem 1.1 of \cite{CHT}, or in \cite{Ko},
extends to show that ${\rm reg}(I^m)$ is a linear function for $m\gg 0$, when $I$ is a homogeneous ideal with  respect to our weighting.

\begin{Remark}
\begin{rm}
Let $B_1 = K[x_1, \ldots, x_s]$ and $B_2 = K[y_1, \ldots, y_s]$
be graded polynomial rings 
with ${\rm wt}(x_i)=d_i$ ($i = 1, \ldots, s$) and
${\rm wt}(y_j)=1$ ($j = 1, \ldots, s$), where the
$d_i$'s are  positive integers satisfying the condition ${\rm gcd}(d_1,\ldots,d_{i-1},d_{i+1},\ldots,d_s)=1$ for $1\le i\le s$.

Let $\delta : B_1 \rightarrow B_2$ be the $K$-algebra graded
homomorphism satisfying $\delta(x_i) = y_i^{d_i}$ for $i = 1, \ldots, s$.
Remark that $\delta$ is flat, $B_2/(x_1, \ldots, x_s)B_2$ is of finite length and
\[
\xi = {\rm max}\{ n \in \ZZ \mid [B_2/(x_1, \ldots, x_s)B_2]_n \ne 0 \} 
= \sum_{i = 1}^s d_i - s .
\]

Let $I$ be a homogeneous ideal of $B_1$.
Then,
\[
(I^m)^{\rm sat} \otimes_{B_1} B_2 = (I^m)^{\rm sat} B_2 = (I^m B_2)^{\rm sat} 
\]
for any $m > 0$.
Thus, using (\ref{basechange}), we have that
\[
{\rm reg}((I^m B_2)^{\rm sat} ) = {\rm reg}((I^m)^{\rm sat}) + \sum_{i = 1}^s d_i - s
\]
for any $m > 0$.
Therefore,
\begin{equation}\label{lim-reg}
\lim_{m \to \infty} \frac{{\rm reg}((I^m B_2)^{\rm sat} )}{m} =
\lim_{m \to \infty} \frac{{\rm reg}((I^m)^{\rm sat} )}{m} .
\end{equation}

Here, set $X = {\rm proj}(B_1)$, $Z = {\rm proj}(B_2)$. By our condition on the $d_i$, there exists a Weil divisor $H$ on $X$ such that ${\mathcal O}_{X}(r)\cong\mathcal O_{X}(rH)$ is a  reflexive, rank 1 sheaf of ${X}_{\PP}$ modules for all $r\in\ZZ$ \cite{Mo}.
Let $Y$ (resp.\ $W$) be the blow-ups of $X$ (resp.\ $Z$) along the ideal sheaf
${\mathcal I} = \tilde{I}$ (resp.\ ${\mathcal I}{\mathcal O}_Z = \widetilde{I B_2}$).
Then, we have the following Cartesian diagram:
\[
\begin{array}{ccc}
W & \longrightarrow & Z \\
{\scriptstyle f} \downarrow \hphantom{{\scriptstyle f}} & & 
\hphantom{{\scriptstyle f}} \downarrow \hphantom{{\scriptstyle f}} \\
Y & \stackrel{\pi}{\longrightarrow} & X 
\end{array}
\]
Set $E = \pi^{-1}({\rm proj}(B_1/I))$.
Let $\ell$ be a positive integer such that ${\mathcal O}_X(\ell)$ is invertible (we can take $\ell={\rm lcm}(d_1,\ldots,d_s)$).
By the projection formula, for any positive integers $\alpha$ and $\beta$,
${\mathcal O}_Y(-\ell \beta E) \otimes \pi^* {\mathcal O}_X(\ell \alpha)$
is nef if and only if so is $f^*({\mathcal O}_Y(-\ell \beta E) \otimes \pi^* {\mathcal O}_X(\ell \alpha))$.
Hence, $\alpha/\beta \ge s_{{\mathcal O}_Z(1)}({\mathcal I}{\mathcal O}_Z)$
if and only if $\alpha/\beta \ge s_{{\mathcal O}_X(1)}({\mathcal I})$.
Therefore, we obtain
\begin{equation}\label{Ses}
s_{{\mathcal O}_Z(1)}({\mathcal I}{\mathcal O}_Z) = 
s_{{\mathcal O}_X(1)}({\mathcal I}) .
\end{equation}

On the other hand, 
\begin{equation}\label{CELformula}
\lim_{m \to \infty} \frac{{\rm reg}((I^m B_2)^{\rm sat} )}{m} = 
s_{{\mathcal O}_Z(1)}({\mathcal I}{\mathcal O}_Z)
\end{equation}
by Theorem~B in \cite{CEL}.
By (\ref{lim-reg}), (\ref{Ses}) and (\ref{CELformula}), we obtain 
\begin{equation}\label{weightedCELformula}
\lim_{m \to \infty} \frac{{\rm reg}((I^m)^{\rm sat} )}{m} = 
s_{{\mathcal O}_X(1)}({\mathcal I}),
\end{equation}
which is the statement of Theorem \ref{Theoremlim}.
\end{rm}
\end{Remark}

\section{Blow ups of a Weighted Projective Plane}\label{sec3}

Suppose that $G$ is a subgroup of $\RR$. Then $G_+$ will denote the semigroup of
positive elements of $G$, and $G_{\ge 0}$ will denote the semigroup of nonnegative elements of $G$.

In this section, we will suppose that $K$ is an algebraically closed field, and $a,b,c\in\ZZ_+$ are pairwise relatively prime. Let
$\PP=\PP(a,b,c)$ be the corresponding weighted projective space.
Suppose that $P_1,\ldots, P_r$ are distinct nonsingular closed points in 
$\PP(a,b,c)$, and $e_1,\ldots,e_r\in \ZZ_+$. 

The coordinate ring of $\PP(a,b,c)$ is the graded polynomial ring
$$
S=K[x,y,z]=\bigoplus_{n\ge 0}H^0(\PP,\mathcal{O}_{\PP}(n)),
$$
which is graded by ${\rm wt}(x)=a$, ${\rm wt}(y)=b$, ${\rm wt}(z)=c$.
Let $\mathfrak m=(x,y,z)$ be the graded maximal ideal of $S$.

Some references on the geometry of weighted projective spaces are \cite{Do} and \cite{Mo}. We have that $\PP(a,b,c)$ is a normal surface, which is nonsingular, except possibly at the three points $Q_1=V(x,y)$, $Q_2=V(x,z)$ and $Q_3=V(y,z)$. 
Since $a,b,c$ are pairwise relatively prime, there exists a Weil divisor $H$ on $\PP$ such that ${\mathcal O}_{\PP}(r)\cong\mathcal O_{\PP}(rH)$ is a  reflexive, rank 1 sheaf of ${\mathcal O}_{\PP}$ modules for all $r\in\ZZ$ \cite{Mo}.
 The canonical divisor on  $\PP$ is
$\mathcal O_{\PP}(-a-b-c)$. We have that $\mathcal O_{\PP}(\ell)$ is an ample invertible sheaf if $\ell={\rm lcm}(a,b,c)$.

Suppose that $L$ is a finitely generated graded $S$-module. Recall (Section \ref{sec2}) that the regularity, ${\rm reg }{L}$, of $L$
is the largest integer $t$ such that there exists an index $j$ such that $H^j_{\mathfrak
m}(L)_{t-j}\ne 0$.  We will denote the sheaf associated to $L$ on $\PP$ by $\tilde L$.

For all $j\in\ZZ$, we have a natural morphism of sheaves of $\mathcal O_{\PP}$ modules 
$$
\Lambda:\tilde L(j):= \tilde L\otimes \mathcal O_{\PP}(j)\rightarrow\widetilde{L(j)}
$$
which is an isomorphism whenever $\mathcal O_{\PP}(j)$ is Cartier.

Let $I_{P_i}\subset K[x,y,z]$ be the weighted homogeneous ideal of the point
$P_i$, and let $I=\cap_{i=0}^rI_{P_i}^{e_i}$.

An important case  is when $i=1$, and
$I=P(a,b,c)$ is the prime ideal of a monomial space curve.

Let ${\mathcal I}=\tilde I$ be the sheaf associated to $I$ on $\PP$. For $m\in\NN$ and
$n\in\ZZ$,
$$
{\mathcal I}^m(n)\cong\widetilde{I^m(n)}\cong {\mathcal I}^m\otimes{\mathcal O}_{\PP}(n),
$$
since ${\mathcal O}_{\PP}(1)$ is locally free at the support of ${\mathcal O}/{\mathcal I}$.  

Let $I^{(m)}$ be the $m$-th symbolic power of $I$, which is also, in our
situation, the saturation $(I^m)^{{\rm sat}}$ of $I^m$ with respect to the
graded maximal ideal $\mathfrak m$ of $S$.
We have, as follows from Theorem A4.1 \cite{E}, that the graded local cohomology satisfies
\begin{equation}\label{eqLC1}
H^0_{\mathfrak m}(I^{(m)})=H^1_{\mathfrak m}(I^{(m)})=0
\end{equation}
for all $m\in\NN$, and for $i\ge 1$, we have graded isomorphisms
\begin{equation}\label{eqLC2}
H^{i+1}_{\mathfrak m}(I^{(m)})\cong \bigoplus_{n\in\ZZ}H^i(\PP,{\mathcal I}^m(n))
\cong \bigoplus_{n\in\ZZ}H^i(\PP,{\mathcal I}^m\otimes{\mathcal O}_{\PP}(n)).
\end{equation}

Let $f:X=X(I)\rightarrow \PP(a,b,c)$ be the blow up of these
points. Let $E_i$ be the exceptional curves mapping to $P_i$ for $1\le i\le r$, and let $E=e_1E_1+\cdots+e_rE_r$. Let $A$ be a Weil divisor  on $X$ such that 
${\mathcal O}_X(A)\cong f^*{\mathcal O}_{\PP}(1)$ (recall that ${\mathcal O}_{\PP}(1)$ is locally free at points where $f$ is not an isomorphism). 

Since $f$ is the blow up of the nonsingular points $P_1,\ldots, P_r$, 
for $m\ge 0$,
$f_*{\mathcal O}_X(-mE)\cong {\mathcal I}^m$ and $R^if_*{\mathcal O}_X(-mE)=0$ for $i>0$ and $m\ge 0$ (for instance by Proposition 10.2 \cite{Ma}). Since ${\mathcal O}_{\PP}(1)$ is locally free above  all points on $\PP$ where $f$ is not an isomorphism, by the projection formula,
$$
f_*{\mathcal O}_X(nA-mE)\cong {\mathcal I}^m\otimes {\mathcal O}_{\PP}(n)\cong 
{\mathcal I}^m(n)
$$
for all $m\in\NN$ and $n\in\ZZ$. By the  Leray spectral sequence, we have that
\begin{equation}\label{eqLC3}
H^i(X,{\mathcal O}_X(nA-mE))\cong H^i(\PP,{\mathcal I}^m(n))
\end{equation}
for all $m\in\NN$, $n\in\ZZ$ and $i \ge 0$.

The $m$-th symbolic power of $I$ can be computed as 
$$
I^{(m)}=\bigoplus_{n\ge 0}H^0(X,\mathcal{O}_X(nA-mE)).
$$

Let $u=\sum_{i=1}^re_i^2$.

We have 
\begin{equation}\label{eq6}
(A\cdot A)=\frac{1}{abc},\, (E_i\cdot E_i)=-1,\, (A\cdot E_i)=0,\, (E\cdot E)=-u\mbox{ and }(A\cdot E)=0.
\end{equation}

Let ${\rm Div}(X)$ be the group of Weil divisors on $X$. 
There is an intersection theory on ${\rm Div}(X)$, developed in \cite{M}, which associates to Weil divisors $D_1$ and $D_2$ on $X$ a rational number $(D_1\cdot D_2)$.
Divisors $D_1$ and $D_2$ are numerically equivalent, written $D_1\equiv D_2$
if $(D_1\cdot C)=(D_2\cdot C)$ for every Weil divisor $C$ on $X$. A $\QQ$-divisor $D$ on $X$ is called numerically ample if $(D\cdot C)>0$ for all curves $C$ on $X$ and $(D\cdot D)>0$. Let $N_1(X)=({\rm Div}(X)/\equiv)\otimes \RR$.  We will write $\overline D$ to denote the class in $N_1(X)$ of a Weil divisor $D$ on $X$.

Let $L$ be the real vector subspace of $N_1(X)$ spanned by (the classes of) $E$ and $A$.
Let $NL=\overline{{\rm NE}}(X)\cap L$ and let $AL=\overline{{\rm AMP}}(X)\cap L$, where $\overline{{\rm NE}}(X)$ is the closure of the cone of curves on $X$, and 
$\overline{{\rm AMP}}(X)$ is the closure of the ample cone on $X$.

We now make a  sketch of the cones $NL$ and $AL$.  
 $NL$ is a cone with boundary   rays $E\RR_{\ge 0}$ and $R=(\tau A-E)\RR_{\ge 0}$ for some $\tau=\tau(I)\in \RR$.
$AL$ is a cone with boundary  rays $A\RR_{\ge 0}$ and $T=(s A-E)\RR_{\ge 0}$, where $s=s(I)=s_{{\mathcal O}_{\PP}(1)}(\mathcal I)\in\RR$ is the $s$-invariant of $I$.

Let $g:Y\rightarrow X$ be the minimal resolution of singularities. Let ${\mathcal L}$ be a line
bundle on $X$. $X$ has rational singularities implies $H^i(Y,{\mathcal M})=H^i(X,{\mathcal L})$
where ${\mathcal M}=g^*({\mathcal L})$. By the Riemann Roch Theorem on $Y$, ($\chi({\mathcal O}_Y)=1$ since $Y$ is rational)
$$
\chi({\mathcal L})=\chi({\mathcal M})=\frac{1}{2}({\mathcal M}\cdot {\mathcal M}\otimes
\omega_Y^{-1})+1.
$$
By the projection formula,
\begin{equation}\label{eq1}
\chi({\mathcal L})=\frac{1}{2}({\mathcal L}\cdot {\mathcal L}\otimes \omega_X^{-1})+1.
\end{equation}

Since $X$ has rational singularities, $g_*\omega_Y=\omega_X={\mathcal O}_X(-(a+b+c)A+E_1+\cdots+E_r)$.

\begin{Proposition}\label{Prop2} We have vanishing of cohomology $H^2(X,\mathcal O_X(\alpha A-\beta E))=0$
if $\beta\ge 0$ and $\alpha>-(a+b+c)$.
\end{Proposition}

\begin{proof}
We have $H^2(X,\mathcal O_X(\alpha A-\beta E))\cong H^2(\PP,{\mathcal I}^{\beta}(\alpha))$. From the exact sequence
$$
0\rightarrow \mathcal I^{\beta}(\alpha)\rightarrow \mathcal O_{\PP}(\alpha)\rightarrow (\mathcal O_{\PP}/\mathcal I^{\beta})(\alpha)\rightarrow 0
$$ 
and the fact that $(\mathcal O_{\PP}/\mathcal I^{\beta})(\alpha)$ has zero dimensional support, we have that
$$
H^2(\PP,{\mathcal I}^{\beta}(\alpha))\cong H^2(\PP,\mathcal O_{\PP}(\alpha))=0
$$
for $\alpha>-(a+b+c)$.
\end{proof}

 From (\ref{eq6}), (\ref{eq1}) and Proposition \ref{Prop2} we see that 
$\sqrt{abcu}A-E\in NL$, since 
$$
((\sqrt{abcu}A-E)\cdot (\sqrt{abcu}A-E))=0.
$$
 Thus 
\begin{equation}\label{eq15}
0<\tau(I)\le \sqrt{abcu}\le s(I).
\end{equation}

Suppose that $D$ is a Weil divisor on $X$. $g^*(D)$ is defined in \cite{M}  as the $\QQ$-divisor on $Y$ which agrees with the
strict transform of $D$ away from the exceptional locus of $g$, and has intersection number 0 with
all exceptional curves. If $F=\sum \alpha_i E_i$ is a $\QQ$-divisor on $Y$ (with $\alpha_i\in
\QQ$), then we define a $\ZZ$-divisor by $\lfloor F\rfloor=\sum \lfloor \alpha_i\rfloor E_i$. If ${\mathcal L}$ is a line bundle on $X$, then as for
instance follows from the projection formula of Theorem 2.1 of \cite{S},
\begin{equation}\label{eq4}
H^0(X,{\mathcal O}_X(D)\otimes{\mathcal L})=H^0(Y,{\mathcal O}_Y(\lfloor g^*(D)\rfloor )\otimes g^*{\mathcal L}).
\end{equation}

\begin{Lemma}\label{Lemma1} Suppose ${\mathcal F}$ is a coherent sheaf on $X$ and  $\mathcal B$ is a
line bundle on $X$.
\begin{enumerate}
\item[1.] The Euler characteristic $\chi({\mathcal F}\otimes {\mathcal B}^n)$ is a polynomial in
$n$ for $n\in\NN$. \item[2.] If $(\mathcal B\cdot {\mathcal O}_{X}(A))>0$ then $H^2(X,{\mathcal
F}\otimes{\mathcal B}^n)=0$ for $n\gg0$.
\end{enumerate}
\end{Lemma}

\begin{proof}  We first prove 1. Let ${\mathcal M}$ be an ample line bundle on the projective surface $X$.
Thus there is a composition series of $\mathcal F$ by ${\mathcal O}_X$ modules ${\mathcal
O}_{Z_i}\otimes{\mathcal M}^{e_i}$ with $1\le i\le m$, where $m$ is a positive integer, $Z_i$ are
(integral) subvarieties of $X$ and  $e_i\in\ZZ$ (c.f. Section 7 of Chapter 1 of \cite{H}). Thus
\begin{equation}\label{eq2}
\chi({\mathcal F}\otimes {\mathcal B}^n) =\sum_{i=1}^m\chi({\mathcal O}_{Z_i}\otimes{\mathcal
M}^{e_i}\otimes{\mathcal B}^n). \end{equation}

If $Z_i=X$ we have that $$\chi({\mathcal O}_{Z_i}\otimes{\mathcal M}^{e_i}\otimes{\mathcal B}^n)$$
is polynomial in $n$ by the Riemann Roch formula (\ref{eq1}). If $Z_i$ is an (integral) curve, then
we have an exact sequence
$$
0\rightarrow {\mathcal O}_{Z_i}\rightarrow {\mathcal O}_{\overline Z_i}\rightarrow {\mathcal
G}_i\rightarrow 0
$$
of ${\mathcal O}_{Z_i}$ modules, where $\overline Z_i$ is the normalization of $Z_i$. Since
${\mathcal G}_i$ has finite support, $\chi({\mathcal O}_{Z_i}\otimes{\mathcal
M}^{e_i}\otimes{\mathcal B}^n)$ is a polynomial in $n$ by the Riemann Roch theorem on the
nonsingular projective curve $\overline Z_i$. In the case when $Z_i$ is a point, $\chi({\mathcal
O}_{Z_i}\otimes{\mathcal M}^{e_i}\otimes{\mathcal B}^n)=\chi({\mathcal O}_{Z_i})=1$ for all $n$.

Now we prove 2. By the consideration of the composition sequence constructed in the first part of
the proof, we are reduced to showing that  for $1\le i\le m$, $H^2(Z_i,{\mathcal
O}_{Z_i}\otimes{\mathcal M}^{e_i}\otimes{\mathcal B}^{\otimes n})=0$ for $n\gg 0$. If $Z_i=X$ this
follows from Proposition \ref{Prop2}. Otherwise, $Z_i$ has dimension smaller than 2, so the
vanishing must hold.
\end{proof}

\section{Regularity of Symbolic Powers}

We continue with the assumptions of Section \ref{sec3}.

\begin{Proposition}\label{Prop1} There exist positive integers $b_0$ and $t_0$ such that if $D$ is a Weil divisor on $X$ such
that $\overline D$ is in the translation of ${\rm AL}$ by $t_0(b_0abcA-E)+abcA$, then
$H^i(X,D)=0$ for $i>0$.
\end{Proposition}

\begin{proof} Let ${\mathcal F}_m={\mathcal O}_{X}(mA)$ for $0\le m<abc$. There exists a positive integer $b_0$ such that
$b_0abcA-E$ is ample on $X$. Let ${\mathcal M}={\mathcal O}_X(b_0abc A-E)$.

We use the following vanishing theorem, proven in  Theorem 5.1 of \cite{F1}. Let ${\mathcal F}$ be coherent on a projective scheme $Y$, and ${\mathcal M}$ be an
ample line bundle. Then there exists an integer $t$ such that $H^i(Y,{\mathcal F}\otimes{\mathcal
L}\otimes {\mathcal M}^t)=0$ for all nef line bundles ${\mathcal L}$ on $Y$ and for all $i>0$.

Choose $t_0$ so that $t_0$ satisfies the condition for $t$ in the above vanishing theorem, for
${\mathcal F}_m$ with $0\le m<abc$ and for all $i>0$.

Suppose $\overline D$ is in the translation of ${\rm AL}$ by $t_0(b_0abcA-E)+abcA$. Then
$D\sim \alpha A-\beta E+t_0(b_0abcA-E)+abcA$, with $\alpha\ge s\beta$. Expand $\alpha=nabc+m$
with $0\le m<abc$. Then we have
$$
{\mathcal O}_X(D)\cong {\mathcal F}_m\otimes{\mathcal L}\otimes {\mathcal M}^{t_0}
$$
where ${\mathcal L}={\mathcal O}_X((n+1)abcA-\beta E)$ is nef. Thus the conclusions of the
Proposition hold for $D$.
\end{proof}

Suppose that $D$ is a divisor on $X$. Define
$$
\overline D^{\perp}=\{\phi\in N_1(X)\mid (D\cdot\phi)=0\}.
$$
An effective divisor $D$ such that $(D\cdot D)<0$ will be called a negative curve.
An effective divisor $D$ such that $D\sim aA-mE$ for some positive integers $a$ and $m$ will be called an $E$-uniform curve.

\begin{Lemma}\label{Lemma2} Either $T=(\sqrt{abcu}A-E)\RR_{\ge 0}$, or
$T=\overline C^{\perp}\cap AL$, where $C$ is an irreducible negative curve.
\end{Lemma}

\begin{proof} Recall from (\ref{eq15}), that $T=(s(I) A-E)\RR_{\ge 0}$, with $s=s(I)\ge \sqrt{abcu}$.
Suppose that $s>\sqrt{abcu}$. There exists $\alpha\in\QQ$ such that $s>\alpha>\sqrt{abcu}$. Write $\alpha=\frac{c}{d}$ where $c,d\in\ZZ_+$. 
By (\ref{eq6}), (\ref{eq1}) and Proposition \ref{Prop2}, we have that 
$h^0(X,{\mathcal O}_X(m(cA-dE)))>0$ for $m\gg 0$. Thus there exist only a finite number of irreducible curves $C_1,\ldots, C_t$ on $X$ such that $(C_i\cdot(\alpha A-E))<0$.

Suppose that $(C\cdot(s A-E))>0$ for all irreducible curves $C$ on $X$. In particular, $(C_i\cdot (s A-E))>0$ for all $1\le i\le t$. This implies that there exists a real number $\beta$ with $\alpha<\beta<s$ such that
$(C_i\cdot(\beta A-E))>0$ for $1\le i\le t$. If $C$ is an irreducible curve on $X$ other than one of the $C_i$, then we have $(C\cdot(\alpha A-E))\ge 0$ and $(C\cdot A)\ge 0$, so that $(C\cdot(\beta A-E))\ge 0$. Thus $\beta A-E\in AL$,
a contradiction. Thus there exists an irreducible curve $C$ on $X$ such that 
$(C\cdot (s A-E))=0$ (the only irreducible curves $C$ on $X$ with $(C\cdot A)=0$ are the $E_i$).

\end{proof}

\begin{Theorem}\label{Theorem3} Let $u=\sum_{i=1}^re_i^2$. We have that $s(I)\ge \sqrt{abcu}$. If 
$s(I)>\sqrt{abcu}$, then $s(I)$ is a rational number. 
\end{Theorem}

\begin{proof} The proof is immediate from (\ref{eq15}) and Lemma \ref{Lemma2}.
\end{proof}

Negative curves and $E$-uniform curves are defined before Lemma \ref{Lemma2}.

\begin{Lemma}\label{Lemma3} Suppose that there exists an $E$-uniform negative curve $F$. Then $s(I)>\sqrt{abcu}$, and $T=\overline C^{\perp}\cap AL$, where $C$ is an irreducible negative curve in the support of $F$.
\end{Lemma}

\begin{proof} Let $F$ be an $E$-uniform negative curve. $F\sim mA-nE$ for some 
$m,n\in \ZZ_+$. Since $(F^2)<0$, and $F$ is effective, there exists an irreducible curve $C$ in the support of $F$ such that $(C\cdot F)<0$.
Since $((s A-E)\cdot(\tau A-E))\ge 0$ and $\tau<\sqrt{abcu}$, we have $s>\sqrt{abcu}$. We have 
$T=(\alpha A+F)\RR_{\ge 0}$ for some $\alpha>0$.  Since $T$ is a boundary ray of $AL$, for all $\epsilon>0$, there exists an irreducible curve $C_{\epsilon}$ on $X$ such that $(C_{\epsilon}\cdot ((\alpha-\epsilon)A+F))<0$. Since $(C_{\epsilon}\cdot A)\ge 0$, we must have
$(C_{\epsilon}\cdot F)< 0$, so that $C_{\epsilon}$ is in the support of $F$. Since $F$ has only a finite number of irreducible components, we have that $T=\overline C^{\perp}$ for some irreducible component $C$ of $F$.
\end{proof}

\begin{Proposition}\label{Prop3} There exist $t_1>0$ and $m_0>0$ such that $m\ge m_0$ implies there
exists a Cartier divisor $D=\alpha A-mE$ such that $D$ lies between the rays $T$ and the
translation of $T$ by $-t_1A$ such that $h^1(X,{\mathcal O}_X(D))\ne 0$.
\end{Proposition}

\begin{proof}  
Let $\gamma=\sqrt{abcu}A-E$. Let 
$$
d=(e_1+\cdots+e_r)\sqrt{\frac{abc}{u}}.
$$
Observe that $\sqrt{abcu}\in\QQ$ if and only if $d\in\QQ$.
Suppose that $n\in \ZZ_+$ and  $\alpha(n)\in \RR_{\ge 0}$ are such that
$n\gamma-\alpha(n)A$ is  a Cartier  divisor. Then by (\ref{eq6}) and (\ref{eq1}),
\begin{equation}\label{eq7}
\begin{array}{l}
\chi({\mathcal O}_X(n\gamma-\alpha(n)A))\\
\,\,\,=
n\frac{\sqrt{u}}{2\sqrt{abc}}\left((a+b+c)-d-2\alpha(n)\right)
+\frac{1}{2abc}\left(\alpha(n)^2-\alpha(n)(a+b+c)\right)+1 .
\end{array}
\end{equation}

By (\ref{eq15}), we always have $s\ge \tau$, so that we reduce to establishing  the Proposition in the two cases $s=\tau$ and $s>\tau$.

\vskip .2truein
\noindent {\bf Case 1}  Assume that $s=\tau$ so that by (\ref{eq15}),  $T=R=\gamma\RR_{\ge 0}$.
For $n\in\ZZ_+$, choose $\alpha(n)$ in (\ref{eq7}) so that $2abc\le \alpha(n)<3abc$. Then $(a+b+c)-d-2\alpha(n)<0$, so that 
$$
-h^1(X,\mathcal O_X(n\gamma-\alpha(n)A))\le\chi(\mathcal O_X(n\gamma-\alpha(n)A))<0
$$
for $n\gg0$.

\vskip .2truein
\noindent {\bf Case 2}  Assume that the boundary ray 
$R=(\tau A-E)\RR_{\ge 0}$ of $NL$ and the boundary ray  $T=(s A-E)\RR_{\ge 0}$ of $NA$ satisfy $s>\tau$.

We must have $s>\sqrt{abcu}$ with these assumptions, for otherwise,
$\gamma\in AL$, and there
would be an effective divisor  $F\equiv \alpha A-\beta E$ such that 
$\frac{\alpha}{\beta}<\sqrt{abcu}$, so that $F$ is an $E$-uniform negative curve. We
would then have that $(F\cdot\gamma)<0$, a contradiction.

 By Lemma \ref{Lemma2}, we have $T=\overline C^{\perp}$, where $C$ is an irreducible negative curve. Let $p_a(C)$ be the arithmetic genus of $C$. Let $\delta = s A-E$. Let
 $$
 \beta={\rm max}\{0,\frac{1-p_a(C)}{(C\cdot A)}\}.
 $$
  For $n\in \ZZ_+$, let $\alpha(n)$
be such that 
\begin{equation}\label{eq13}
\beta <\alpha(n)\le\beta+abc
\end{equation}
and $n\delta-\alpha(n)A$ is a Cartier divisor.

We have an exact sequence of $\mathcal{O}_X$ modules
\begin{equation}\label{eq14}
0\rightarrow \mathcal{O}_X(n\delta-\alpha(n)A-C)\rightarrow \mathcal {O}_X(n\delta-\alpha(n)A)\rightarrow   \mathcal{ O}_X(n\delta-\alpha(n)A)\otimes\mathcal{O}_C\rightarrow 0.
\end{equation}

Since $s>\sqrt{abcu}$, be have that $h^0(X,\mathcal{O}_X(n\delta-\alpha(n)A))>0$ for $n\gg0$. We further have
that $((n\delta-\alpha(n)A)\cdot C)<0$, so since $C$ is an integral curve, for $n\gg 0$,
$$
h^0(X,\mathcal{O}_X(n\delta-\alpha(n)A-C))=h^0(X,\mathcal{O}_X(n\delta-\alpha(n)A))>0.
$$
We have that 
$$
h^2(X,\mathcal{O}_X(n\delta-\alpha(n)A-C))=h^2(X,\mathcal{O}_X(n\delta-\alpha(n)A))=0
$$
 by Proposition \ref{Prop2}. 

From (\ref{eq14}) we now have
$$
\begin{array}{l}
h^1(X,\mathcal{O}_X(n\delta-\alpha(n)A-C))-h^1(X,\mathcal {O}_X(n\delta-\alpha(n)A))\\
\,\,\,= \chi(\mathcal {O}_X(n\delta-\alpha(n)A))-\chi(\mathcal {O}_X(n\delta-\alpha(n)A-C))\\
\,\,\,= \chi(\mathcal {O}_X(n\delta-\alpha(n)A)\otimes \mathcal{O}_C)\\
=(C\cdot(n\delta-\alpha(n)A))+1-p_a(C)<0
\end{array}
$$
where  the last equality is by the Riemann Roch theorem for the curve $C$, and (\ref{eq13}). Thus   $h^1(X,\mathcal{O}_X(n\delta-\alpha(n)A))>0$
for $n\gg 0$.

\end{proof}

Recall that $\lfloor x\rfloor$ is the greatest integer in a real number $x$.

\begin{Theorem}\label{Theorem4} There exists a bounded function $\sigma_I:\NN\rightarrow \ZZ$ such
that
$$
{\rm reg}(I^{(m)})=\lfloor s(I) m\rfloor +\sigma_I(m)
$$
for all $m\in\NN$.
\end{Theorem}

\begin{proof} For $i\ge 2$,
$$
H^i_{\mathfrak m}(I^{(m)})_{n}=H^{i-1}(X,{\mathcal O}_X(n A-mE))
$$
by equations (\ref{eqLC2}) and (\ref{eqLC3}).
The theorem now follows from Propositions \ref{Prop1}, \ref{Prop2} and \ref{Prop3} for large $m$, and
thus the theorem is true for all $m$.
\end{proof}

A function $\sigma:\NN\rightarrow \ZZ$ is {\it eventually periodic} if $\sigma(m)$ is periodic for
 $m\gg 0$.

Recall (Theorem \ref{Theorem3}) that  $s(I)\ge \sqrt{abcu}$, and if $s(I)>\sqrt{abcu}$, then $s(I)$ is a rational number.

\begin{Theorem}\label{Theorem5}  Suppose that $s(I)>\sqrt{abcu}$
and $K$ has characteristic zero or is an algebraic closure of a finite field.
Then the function $\sigma_I(m)$ of Theorem \ref{Theorem4} is
eventually periodic.
\end{Theorem}

\begin{proof}  By Propositions
\ref{Prop1}, \ref{Prop2} and \ref{Prop3}, we need only compute $h^1(X,{\mathcal O}_X(n A-mE))$ for $n
A-mE$ between the rays $T$ translated up by $(t_0b_0+1)abcA$ and $T$ translated  down by $-t_1A$.
Call this region $\Delta$.

By Theorem \ref{Theorem3}, there exists a numerically effective Cartier divisor $G$ such that $T=\overline G\RR_{\ge 0}$. We have $(G^2)>0$.  Since $\overline G$ is
rational, there exists a finite number of Weil divisors $D_i$ with $\overline D_i\in\Delta$ such that every divisor $D$ with 
$\overline D\in\Delta$ can be written as $D\sim D_i+nG$ for some $i$ and some $n\in \NN$. Let ${\mathcal
L}={\mathcal O}_X(G)$.

Since $(g^*{\mathcal O}_{X}(abc)\cdot g^*{\mathcal L})>0$, Serre duality on $Y$ implies
\begin{equation}\label{eq5}
h^2(Y,{\mathcal O}_Y(\lfloor g^*(D_i)\rfloor)\otimes g^*{\mathcal L}^n)=0
\end{equation}
for all $i$ and for $n\gg0$.

By 2 of Lemma \ref{Lemma1}, (\ref{eq4}) and (\ref{eq5}),  for all $i$ and for $n\gg0$ we have
$$
\begin{array}{lll}
h^1(X,{\mathcal O}_X(D_i)\otimes{\mathcal L}^n)&=&h^0(X,{\mathcal O}_X(D_i)\otimes{\mathcal
L}^n)-\chi({\mathcal O}_X(D_i)\otimes{\mathcal L}^n)\\
&=& h^0(Y,{\mathcal O}_Y(\lfloor g^*(D_i)\rfloor)\otimes g^*({\mathcal L}^n))-\chi({\mathcal O}_X(D_i)\otimes{\mathcal L}^n)\\
&=&h^1(Y,{\mathcal O}_Y(\lfloor g^*(D_i)\rfloor)\otimes g^*({\mathcal L}^n))\\
&&\,\,\,+\chi({\mathcal O}_Y(\lfloor g^*(D_i)\rfloor )\otimes g^*({\mathcal L}^n)) - \chi({\mathcal
O}_X(D_i)\otimes{\mathcal L}^n).
\end{array}
$$
By the Riemann Roch Theorem on $Y$ and 1 of Lemma \ref{Lemma1} we have that
$$
\chi({\mathcal O}_Y(\lfloor g^*(D_i)\rfloor)\otimes g^*({\mathcal L}^n)) - \chi({\mathcal
O}_X(D_i)\otimes{\mathcal L}^n)
$$
is a polynomial in $n$ for all $i$.

By Proposition 13 of \cite{CS}, there exists an effective divisor $C$ on $Y$ such that $g^*({\mathcal L})\otimes{\mathcal O}_C$ is numerically trivial, and the restriction maps
$$
H^1(Y,{\mathcal O}_Y(\lfloor g^*(D_i)\rfloor)\otimes g^*({\mathcal L}^n))
\rightarrow 
H^1(C,{{\mathcal O}_C\otimes \mathcal O}_Y(\lfloor g^*(D_i)\rfloor)\otimes g^*({\mathcal L}^n))
$$
are isomorphisms for $n\gg 0$ and all $i$.

In the case when $K$ has characteristic zero, Theorem 8   of \cite{CS}, shows that for all $i$,
$h^1(C,{{\mathcal O}_C\otimes \mathcal O}_Y(\lfloor g^*(D_i)\rfloor)\otimes g^*({\mathcal L}^n))
$ is eventually periodic in $n$ for all $i$. In the case when $K$ is an algebraic closure of a finite field, then the numerically trivial invertible sheaf ${\mathcal O}_C\otimes g^*({\mathcal L})$ must be torsion, so some power
is isomorphic to ${\mathcal O}_C$. Thus we trivially have that $h^1(C,{{\mathcal O}_C\otimes \mathcal O}_Y(\lfloor g^*(D_i)\rfloor)\otimes g^*({\mathcal L}^n))
$ is eventually periodic in $n$ for all $i$.

In either case of $K$,
$h^1(Y,{\mathcal O}_Y(\lfloor g^*(D_i)\rfloor)\otimes g^*({\mathcal L}^n))$ is eventually periodic as a function
of $n$. Thus $\sigma(m)$ is eventually periodic. 
\end{proof}
 When $K$ is a field of positive characteristic which has positive transcendence degree over the prime field,
the conclusions of Theorem \ref{Theorem5} may fail. An example of a set of points in ordinary projective space $\PP^2$ where $\sigma(m)$
is not eventually periodic is given in Example 4.4 \cite{CHT}.

\begin{Corollary}\label{Cor6} Suppose  there exists an $E$-uniform negative curve,
and $K$ has characteristic zero, or is a finite field. Then the function $\sigma_I(m)$ of Theorem \ref{Theorem4} is
eventually periodic.
\end{Corollary}

\begin{proof} This follows from Lemma \ref{Lemma3} and Theorem \ref{Theorem5}.
\end{proof}

An important case of this construction is when $i=1$, and
$I=P(a,b,c)$ is the prime ideal of a monomial space curve. 
As a corollary to Theorems \ref{Theorem4} and \ref{Theorem5}, we have the following application  to  monomial space curves.

\begin{Corollary}\label{Cor7*}  Suppose that $I=P(a,b,c)$ is the prime ideal of a monomial space curve, and   there exists a  negative  curve on $X(I)$.
 Then $s(I)$ is a rational number, and the function $\sigma_I(m)$ of Theorem \ref{Theorem4} is
eventually periodic. 
\end{Corollary}

\section{Uniform negative curves and Nagata's conjecture}\label{sec5}

Let $S$ be a polynomial ring as in Section \ref{sec3}.
Let $C = K[u,v,w]$ be a polynomial ring with 
${\rm wt}(u) = {\rm wt}(v) = {\rm wt}(w) = 1$.

Consider the $K$-algebra homomorphism
\[
\delta : S \longrightarrow C
\]
defined by $\delta(x) = u^a$, $\delta(y) = v^b$, 
$\delta(z) = w^c$, where $a$, $b$, $c$ are pairwise relatively prime 
positive integers.

Let
\[
I = I_{P_1}^{e_1} \cap \cdots \cap I_{P_r}^{e_r}
\]
be an ideal of $S$ as in Section \ref{sec3}.

Consider the following two conditions:
\begin{itemize}
\item[(A1)]
$K$ is an algebraically closed field such that ${\rm ch}(K)$ is $0$ or ${\rm ch}(K)$ does not divide $abc$.
\item[(A2)]
$I_{P_i} \not\ni xyz$ for $i = 1, \ldots, r$.
\end{itemize}

\begin{Lemma}\label{split}
Assume the conditions (A1) and (A2) as above.
There are distinct prime ideals $Q_{i1}$, $Q_{i2}$, \ldots, $Q_{i,abc}$
of $C$ such that
\[
I_{P_i}^{(m)}C = \bigcap_{j = 1}^{abc} Q_{ij}^{(m)}
\]
for $m > 0$ and $i = 1, \ldots, r$.
\end{Lemma}

\begin{proof}
Let $Q$ be a prime ideal of $C$ lying over $I_{P_i}$.
Then, there exists a point $(\alpha:\beta:\gamma) \in {\Bbb P}_K^2$
such that
\[
Q = I_2
\left(
\begin{array}{ccc}
u & v & w \\
\alpha & \beta & \gamma
\end{array}
\right) ,
\]
where $I_2( \ )$ is the ideal generated by all the $2 \times 2$-minors
of the given matrix.
Remark that $Q$ is the kernel of the $K$-algebra homomorphism
\[
\phi_{(\alpha:\beta:\gamma)} : C \longrightarrow K[t]
\]
defined by  $\phi_{(\alpha:\beta:\gamma)}(u) = \alpha t$,
$\phi_{(\alpha:\beta:\gamma)}(v) = \beta t$,
$\phi_{(\alpha:\beta:\gamma)}(w) = \gamma t$.
Then, $I_{P_i}$ is the kernel of the $K$-algebra homomorphism
\[
\phi = \phi_{(\alpha:\beta:\gamma)} \delta : S \longrightarrow K[t]
\]
defined by $\phi(x) = \alpha^a t^a$,
$\phi(y) = \beta^b t^b$,
$\phi(z) = \gamma^c t^c$.
Let $\zeta_q$ be a primitive $q$-th root of $1$ for a positive integer $q$.

Set
\[
Q_{n_1,n_2,n_3} = I_2
\left(
\begin{array}{ccc}
u & v & w \\
\zeta_a^{n_1}\alpha & \zeta_b^{n_2}\beta & \zeta_c^{n_3}\gamma
\end{array}
\right) .
\]
It is the kernel of the $K$-algebra homomorphism $\phi_{(\zeta_a^{n_1}\alpha:\zeta_b^{n_2}\beta:\zeta_c^{n_3}\gamma)}$.
For any $n_1$, $n_2$ and $n_3$, 
$Q_{n_1,n_2,n_3}$ is a prime ideal of $C$ lying over $I_{P_i}$
since $\phi_{(\zeta_a^{n_1}\alpha:\zeta_b^{n_2}\beta:\zeta_c^{n_3}\gamma)} \delta = \phi$.
By our assumption (A2), 
all of $\alpha$, $\beta$ and $\gamma$ are not zero.
By (A1),
\[
\{ (\zeta_a^{n_1}\alpha : \zeta_b^{n_2}\beta : \zeta_c^{n_3}\gamma) 
\in {\Bbb P}^2_K \mid
n_1 = 0, \ldots, a-1; \ 
n_2 = 0, \ldots, b-1; \ 
n_3 = 0, \ldots, c-1
\}
\]
are distinct $abc$ points in ${\Bbb P}^2_K$.
Therefore,
\[
\{ 
Q_{n_1,n_2,n_3} \mid
n_1 = 0, \ldots, a-1; \ 
n_2 = 0, \ldots, b-1; \ 
n_3 = 0, \ldots, c-1
\}
\]
are distinct prime ideals of $C$ lying over $I_{P_i}$.
Here, we have
\begin{eqnarray*}
& & {\rm Ass}_C(C/I_{P_i}C) \\
& = & {\rm Min}_C(C/I_{P_i}C) \\
& = & \{ Q \in {\rm Spec}(C) \mid Q \cap S = I_{P_i} \} \\
& \supseteq & \{ 
Q_{n_1,n_2,n_3} \mid
n_1 = 0, \ldots, a-1; \ 
n_2 = 0, \ldots, b-1; \ 
n_3 = 0, \ldots, c-1
\} .
\end{eqnarray*}
Since $C$ is an $S$-free module of rank $abc$,
$C/I_{P_i}C$ is an $S/I_{P_i}$-free module of rank $abc$.
Then, it is easy to see
\[
abc = {\rm rank}_{S/I_{P_i}}(C/I_{P_i}C)
= \sum_{\scriptstyle
\begin{array}{c}
{\scriptstyle Q \in {\rm Spec}(C)} \\
{\scriptstyle Q \cap S = I_{P_i}}
\end{array}
}
\ell_C(C_Q/I_{P_i}C_Q)  \cdot {\rm rank}_{S/I_{P_i}}(C/Q) .
\]
Therefore,
\[
\{ Q \in {\rm Spec}(C) \mid Q \cap S = I_{P_i} \} =
\{ 
Q_{n_1,n_2,n_3} \mid
n_1 = 0, \ldots, a-1; \ 
n_2 = 0, \ldots, b-1; \ 
n_3 = 0, \ldots, c-1
\}
\]
and 
\[
\ell_C(C_{Q_{n_1,n_2,n_3}}/I_{P_i}C_{Q_{n_1,n_2,n_3}}) = 1
\]
for each $n_1$, $n_2$, $n_3$.
It follows from the equation as above that 
\[
 I_{P_i}C_{Q_{n_1,n_2,n_3}} = Q_{n_1,n_2,n_3}C_{Q_{n_1,n_2,n_3}}
\]
for each $n_1$, $n_2$, $n_3$.

Since
\begin{eqnarray*}
& & {\rm Ass}_C(C/I_{P_i}^{(m)}C)  \\
& =  & {\rm Min}_C(C/I_{P_i}^{(m)}C)  \\
& =  & {\rm Min}_C(C/I_{P_i}C)  \\
& = & \{ 
Q_{n_1,n_2,n_3} \mid
n_1 = 0, \ldots, a-1; \ 
n_2 = 0, \ldots, b-1; \ 
n_3 = 0, \ldots, c-1
\} ,
\end{eqnarray*}
we have
\[
I_{P_i}^{(m)}C 
= 
\bigcap_{n_1,n_2,n_3} 
\left(
I_{P_i}^{(m)}C_{Q_{n_1,n_2,n_3}} \cap C
\right)
= 
\bigcap_{n_1,n_2,n_3} 
\left(
Q_{n_1,n_2,n_3}^mC_{Q_{n_1,n_2,n_3}} \cap C
\right)
=
\bigcap_{n_1,n_2,n_3} Q_{n_1,n_2,n_3}^{(m)}
\]
\end{proof}

By Lemma~\ref{split},
\[
I^{(m)}C = \bigcap_{i = 1}^rI_{P_i}^{(e_im)}C
= \bigcap_{i = 1}^r(\bigcap_{j = 1}^{abc} Q_{ij}^{(e_im)}) .
\]
In this case,
\[
{\rm max}\{ n \in \ZZ \mid [C/(u^a,v^b,w^c)C]_n \ne 0 \} 
= a+b+c-3.  
\]
Then, by (\ref{basechange}),
\begin{equation}\label{reduce}
{\rm reg}(I^{(m)}C) = {\rm reg}(I^{(m)}) + a+b+c-3 .
\end{equation}

By (\ref{reduce}), we may assume $a = b = c = 1$
in Theorem~\ref{Theorem4} and Theorem~\ref{Theorem5}
if (A1) and (A2) are satisfied.

\vspace{2mm}

Let $q_1$, \ldots, $q_n$ be independent generic points in ${\Bbb P}^2_{\Bbb C}$.
Suppose that $n \ge 10$.
Nagata conjectured that 
\[
[I_{q_1}^m \cap \cdots \cap I_{q_n}^m]_d = 0
\]
if $d \le \sqrt{n}m$.
Nagata~\cite{N1} solved it affirmatively when $n$ is a square.

Consider the following two conditions:
\begin{itemize}
\item[(A0)]
$K = \CC$, the field of complex numbers, and
\item[(A3)]
$I = \sqrt{I}$, that is $e_1 = e_2 = \cdots = e_r = 1$ and $E = E_1 + \cdots + E_r$.
\end{itemize}

\begin{Proposition}\label{Nagata}
Suppose that $n$ is a positive integer which has a factorization $n=abcr$ by positive integers with $a,b,c$ 
pairwise relatively prime.
 If there exist distinct points $P_1,\ldots, P_r$ on the weighted projective space 
$\PP_{\CC}(a,b,c)$ satisfying $(A2)$, such that there does not exist an $E$-uniform negative curve on the blow up of $\PP_{\CC}(a,b,c)$
defined by $(A3)$, then Nagata's conjecture is true for $abcr$ general points in $\PP^2_{\CC}$.
\end{Proposition}

\begin{proof}
Assume that Nagata's conjecture is not true for $abcr$ general 
points in ${\Bbb P}^2_{\Bbb C}$.
If $abcr$ is a square, Nagata solved the conjecture affirmatively.
Therefore, we may assume that $\sqrt{abcr}$ is not a rational number.

Let $q_1$, \ldots, $q_{abcr}$ be independent generic points in 
${\Bbb P}^2_{\Bbb C}$.
$I_{q_i}$ is the defining ideal of $q_i$.
By our assumption, there exist positive integers $m_0$ and $d_0$ such that
\begin{equation}\label{NC}
d_0 \le \sqrt{abcr}m_0 \mbox{ \ \ and \ \ }
[I_{q_1}^{m_0} \cap \cdots \cap I_{q_{abcr}}^{m_0}]_{d_0} \ne 0 .
\end{equation}
Since $\sqrt{abcr}$ is not a rational number,
we have
\[
d_0 < \sqrt{abcr}m_0 .
\]

Assume that there does not exist an $E$-uniform negative curve 
for some $P_i$'s satisfying (A0), (A2) and (A3).
Then we have
\[
[I^{(m)}]_d = [I_{P_1}^{(m)} \cap \cdots \cap I_{P_r}^{(m)}]_d = 0
\]
if $d < \sqrt{abcr}m$.
Considering the $\CC$-algebra homomorphism
\[
\delta : S=\CC[x,y,z] \longrightarrow C= \CC[u,v,w] ,
\]
we obtain 
\[
0 = [I^{(m)}C]_d = [\bigcap_{i = 1}^rI_{P_i}^{(m)}C]_d
= [\bigcap_{i = 1}^r(\bigcap_{j = 1}^{abc} Q_{ij}^{(m)})]_d
\]
for $d < \sqrt{abcr}m$.
This contradicts to (\ref{NC}) since we can specialize
$\{ q_1, \ldots, q_{abcr} \}$ to
\[
\{ Q_{ij} \mid i = 1, \ldots, r; \ j = 1, \ldots, abc \} .
\]
\end{proof}

\begin{Remark}
\begin{rm}
Let $K$ be a field and $a$, $b$, $c$ be pairwise relatively prime integers.
Let $P_K(a,b,c)$ be the kernel of the $K$-algebra homomorphism
\[
\delta : S = K[x,y,z] \rightarrow K[t]
\]
defined by $\delta(x) = t^a$, $\delta(y) = t^b$, $\delta(z) = t^c$.

Let $X_K(a,b,c)$ be the blow-up of the weighted projective space $\PP(a,b,c)$
at the point corresponding to $P_K(a,b,c)$.

Assume that $K$ is of positive characteristic.
If there exists a negative curve on $X_K(a,b,c)$,
then the symbolic Rees ring
\begin{equation}\label{Rees}
S \oplus P_K(a,b,c) \oplus P_K(a,b,c)^{(2)} \oplus 
P_K(a,b,c)^{(3)} \oplus \cdots
\end{equation}
is Noetherian by \cite{C}.

Here, assume that the symbolic Rees ring (\ref{Rees})
is not Noetherian for some $a_0$, $b_0$, $c_0$ over some field $K_0$ 
of positive characteristic.
Since $P_{K_0}(a,b,c)$ is not a complete intersection,
we may assume $3 \le a_0 < b_0 < c_0$.
In particular, $a_0b_0c_0 \ge 60 > 10$.
Then by \cite{C}, there is no negative curve on $X_{K_0}(a_0,b_0,c_0)$.
By a standard method of mod $p$ reduction,
there is no negative curve on $X_{\CC}(a_0,b_0,c_0)$.
Then, by Proposition~\ref{Nagata},
Nagata's conjecture is true for $a_0b_0c_0$.
\end{rm}
\end{Remark}

\end{document}